\documentclass[12pt]{amsart}

\usepackage[english]{babel}
\usepackage[T1]{fontenc}
\usepackage[latin1]{inputenc}
\usepackage{lmodern}
\usepackage{a4}
\usepackage[T1]{fontenc}
\usepackage[all]{xy}
\usepackage{verbatim}
\usepackage{stmaryrd}
\usepackage{amsfonts}
\usepackage{amsmath,amssymb,amsthm}
\usepackage{enumerate}
\usepackage{xspace}
\usepackage{ulem}
\usepackage{euscript}
\usepackage{epsfig}
\usepackage{fancyhdr}
\usepackage{txfonts}
\usepackage{babel,indentfirst}

\makeatletter \@addtoreset{equation}{section}

\usepackage[all]{xy}

\theoremstyle{definition}

\newtheorem{thm}{Theorem}[section]

\newtheorem{lem}{Lemma}[section]
\newtheorem{pro}{Proposition}[section]

\newtheorem{cor}{Corollary}[section]

\newtheorem{rem}{Remark}[section]

\newcommand{\B}{\mathcal{B}}

\newcommand{\C}{\mathbb{C}}

\newcommand{\R}{\mathbb{R}}

\title{ Commuting maps   with the Mean Transform under Jordan product}
\author{F. Chabbabi }

\address{Department of Mathematics, FS, Abdelmalek Essaadi University, Tetouan, Morocco}
\email{f.chabbabi@uae.ac.ma}



\date{}

\begin{document}
\maketitle

\begin{abstract}

 In this article, we give a complete characterization of the bijective maps   which  commute with the mean transform under Jordan product. The main result is the following : 
 Let $H,K$ be two complex Hilbert spaces  and  $\Phi :\ B(H) \to \B(K)$ be a bijective map, then 
   $$ \mathcal {M}(\Phi(A)\circ\Phi(B))=\Phi(\mathcal{M}(A\circ B)) \;\; \text{for all}\;\; A, B \in \B(H)$$
   if and only if there exists a unitary or anti-unitary  operator $U:H\to K $ such that, 
   $$  \Phi(T)= UTU^* \; \text{for all} \;T\in \B(H).$$
\end{abstract}

\section{Introduction} 
   Let $H$ and $K$ be  two complex Hilbert spaces  and let  $\mathcal{B}(H, K)$ be the Banach space of all bounded linear operators from $H$ into $K$. In the case
  $K=H$, $\mathcal{B}(H,H)$  is simply denoted by  $\B(H) $ and it  is a  Banach algebra.
 
  For an arbitrary operator $T\in \B(H, K)$, we denote by $\mathcal{R}(T)$, $\mathcal{N}(T)$ and  $T^*$  the range,  the null subspace and   the adjoint operator  of $T$,   respectively. 
  For $T\in \B(H)$, the spectrum of $T$ is denoted by $\sigma(T)$.
 
 An operator $T\in \mathcal{B}(H,K)$ is  a {\it partial isometry} when $T^*T$ is an orthogonal projection (or, equivalently $TT^*T = T$). 
  In particular $T$ is an {\it isometry} if $T^*T=I$, and $T$ is {\it unitary} if it is a surjective isometry.  An operator $T\in \B(H)$ is said to be normal if $T^*T=TT^*$, and quasi-normal if $TT^*T=T^*TT$.
 As usual,  we denote the module of $T \in \B(H)$ by $|T|=(T^*T)^{1/2}$, and $T=V|T|$ is   the unique polar decomposition  of $T$, where $V$ is a partial isometry satisfying  $\mathcal{N}(V)=\mathcal{N}(T)$. In general $V$ and $|T|$ does not commute and it is the case if $T$ is quasi-normal.   
 
From the polar decomposition, the Aluthge transform is defined in (\cite{alu, oku}), by 
$$\Delta(T)=|T|^{\frac{1}{2}}V|T|^{\frac{1}{2}}.$$
The Aluthge transform has been well  studied by many authors, it is a good tool for studying sum class of  operators (see \cite{cha, fm, kp3,kp2,yam}). In the same way, the mean transform of the operator $T$ was  introduced  in \cite{LLL, hyy}, by 
$$\mathcal{M}(T):=\frac{1}{2}(V|T|+|T|V).$$
The mean transform has also well been  studied in many articles, it is the arithmetic mean of $T=V|T|$ and it's Duggal transform $\tilde{T}=|T|V$, it has nice properties, for example we can cite \cite{ccm, hyy, jkp}. The fixed point of mean transform are the quasi-normal operator as the Aluthge transform. 

The main result of the present paper is the following theorem which gives a nice characterization of the commuting maps with the mane transform under the Jordan product.

\begin{thm} \label{th1}
 Let $H$ and $K$ be  two  complex Hilbert space, with $\dim H \geq 3$. Let   $\Phi :\B(H)\to\B(K)$ be a bijective map. Then $\Phi$ satisfies the condition 
\begin{equation*}
\mathcal {M}(\Phi(A)\circ\Phi(B))=\Phi(\mathcal{M}(A\circ B)) \;\; \text{for all}\;\; A, B \in \B(H),
\end{equation*}
  if and only if  there exists  a unitary or anti-unitary operator $U : H\to K $,  such that  
 $$\Phi(T)= UTU^*\;\;\;\;\text{ for every} \; T\in\B(H).$$
\end{thm}
 
\begin{rem}
 Observe that,  even if the  hypothesis on the map $\Phi$  is purely algebraic,  the conclusion gives automatically the continuity of the map.
   Also, the linearity of $\Phi$ is not assumed, we get it automatically.
\end{rem}

\section{Auxiliary results of the mean transform}

   For $x,y \in H$  we denote by  $x\otimes y$ the  rank one operator (or $0$) defined by 
 $$(x\otimes y)u = <u,y>x \: \mbox{  for } \:  u\in H.$$
  Every rank one   operator has  the previous form, and  $x\otimes y$ is an orthogonal projection (i.e. $T^2=T=T^*)$ if and only  if $x=y$ and $\|x\|=1$.   

 The following results can be found in \cite{ccm}.
  \begin{pro}\label{p1}\cite{ccm} Let $x,y \in H$ be two non-zero vectors. Let $T=x\otimes y$, then 
$$\mathcal{M}({T})=\mathcal{M}({x\otimes y})=\frac{1}{2}(x+\dfrac{<x,y>}{\|y\|^2}y)\otimes y.$$  
\end{pro}

\begin{pro}\label{pp2}\cite{ccm} Let $T\in B(H)$. Then 
$$\mathcal{M}(T)=0 \;\;\iff T=0.$$
\end{pro}
Next lemma gives a characterization of the nilpotent operator of order two.
  \begin{lem}\cite{ccm} \label{L1}
 Let $T\in \B(H)$. Then 
$$ \mathcal{M}(T)=\frac{T}{2} \: \text{ if and only if} \: T^2 = 0.$$
\end{lem}
\begin{thm}\label{thm1} Let $T\in \B(H)$. Then 
$$ \mathcal{M}(T^2)=T \: \text{ if and only if  $T$ is an orthogonal projection}.$$
\end{thm}
\begin{proof} It is clear that, if $T$ is a projection, then $\mathcal{M}(T^2)=T=T^*$. So we need to show the direct meaning.
 
Let $T^2=V_2|T^2|$ be the polar decomposition of $T^2$ and suppose that $ \mathcal{M}(T^2)=T$. We will divided the proof on a few steps. 

\textbf{Step 1 : } The operator $|T^2|V_2|T^2|$ is positive. From the assumption, we have 
\begin{equation}\label{L7*}
\dfrac{T^2+|T^2|V_2}{2}=T \;\;\text{ and }\;\;\dfrac{(T^2)^*+V_2^*|T^2|}{2}=T^*.
\end{equation}
Thus 
$$\dfrac{|T^2|+V_2^*|T^2|V_2}{2}=V_2^*T \geq 0, $$
 is fact that the operators 
$|T^2|$ and $V_2^*|T^2|V_2$ are positive. Therefore $$T^*|T^2|=T^*V_2^*T^2=T^*(V_2^*T)T\geq 0.$$
  In particular $T^*|T^2|$ is self adjoint. Hence 
$T^*|T^2|=|T^2|T$. Multiplying on the left by $T^*$, we get that   
$$|T^2|V_2|T^2|=|T^2|TT=T^*|T^2|T\geq 0.$$

\textbf{Step 2 : } $2I-T^*$ is injective. 

Let $x\in H$ such that $(2I-T^*)x=0$. Then  $T^*x=2x$ and  thus $(T^2)^*x=4x$. By the assumption we have  $|T^2|V_2=2T-T^2$. We take the adjoint, we  get that 
$$V_2^*|T^2|=2T^*-(T^2)^*.$$ 
It follows that $V_2^*|T^2|x=2T^*x-(T^2)^*x=0$. Thus $|T^2|V_2^*|T^2|x=0$. Since  
$$|T^2|V_2^*|T^2|=|T^2|V_2|T^2|\geq 0 \;\; \text{ ( see  Step 1)},$$
 then $$|T^2|V_2|T^2|x=0.$$  
 Multiplying the preceding equality on the left by $V_2$,  we get that 
 $$T^4x=V_2|T^2|V_2|T^2|x=0.$$ 
 It follows that 
 $$16\|x\|^2=<(T^4)^*x,x>=<x,T^4x>=0.$$
  By consequence $x=0$ and then $2I-T^*$ is injective. 
  
\textbf{Step 3 : } The inclusion $\mathcal{N}(T^2)\subseteq\mathcal{N}(T^*)\subseteq \mathcal{N}((T^2)^*)$ holds. 

Let $x\in \mathcal{N}(T^2)$, we have  $T^2x=|T^2|x=0$, hence by equation (\ref{L7*})  $$V_2^*|T^2|x=(2I-T^*)T^*x=0,$$ 
and since $2I-T^*$ is injective then $T^*x=0$, thus $x\in \mathcal{N}(T^*)\subseteq \mathcal{N}((T^2)^*)$.

\textbf{Step 4 : } $T$ is a projection.

Now we are in position to prove that $T$ is a projection. First we show that $V_2\geq 0$.  Let $x\in H$ and   consider the following decomposition  $H=\mathcal{N}(|T^2|)\oplus \overline{\mathcal{R}(|T^2|)}$. Put  
$x=x_1+x_2$ with $x_1\in \mathcal{N}(|T^2|), \;\; x_2\in \overline{\mathcal{R}(|T^2|)} $. Then we have
\begin{eqnarray*}
\langle V_2x,x\rangle &=& \langle V_2x_1+V_2x_2,x_1+x_2\rangle \\
&=&\langle V_2x_2,x_2 \rangle +\langle V_2x_2,x_1\rangle \\
&=& <V_2x_2,x_2>+<x_2,V_2^*x_1> \\
& = &<V_2x_2,x_2> \;\;(\text{ since } \;\; \mathcal{N}(V_2)\subseteq  \mathcal{N}(V_2^*)).
\end{eqnarray*}
Since  $|T^2|V_2|T^2|\geq 0$ ( see Step  1),  $<V_2x,x>=<V_2x_2,x_2>\geq 0$. Thus  $V_2$ a positive partial isometry,  so $V_2$ is an orthogonal projection and 
$T^2=V_2|T^2|=V_2^*V_2|T^2|=|T^2|$ is positive. Therefore $\mathcal{M}(T^2)=T^2=T=T^*$. This completes the proof.  
\end{proof}


\begin{cor} Let $T\in \B(H)$. Then the following equivalence holds: 
$$T \;\text{ is a projection } \;\iff (\mathcal{M}(T))^2=T.$$
\end{cor}

\begin{proof}
We need to show the second direction. Suppose that 
$(\mathcal{M}(T))^2=T$. Then  $\mathcal{M}((\mathcal{M}(T))^2)=\mathcal{M}(T)$. By Theorem \ref{thm1} replacing $T$ by $\mathcal{M}(T)$, we get  that $\mathcal{M}(T)$ is a projection. Hence $\mathcal{M}(T)= (\mathcal{M}(T))^2=T$ is also a projection. 
\end{proof}
\begin{lem}\cite{ccm} Let $T\in \B(H)$ and $T=V|T|$ be the polar decomposition of $T$. The following equivalence holds: 
$$\mathcal{M}(T) \;\;\text{ is self-adjoint }\;\iff V \;\;\text{  is self-adjoint}. $$ 
\end{lem}

\begin{lem}\label{Lem1} Let $T\in \B(H)$. Then
$\mathcal{M}(T)$ is a projection if and only if  $T$  is also a projection. In this case  $\mathcal{M}(T)=T$.
\end{lem}
\begin{proof}
It is clear that we need to show the first direction. Let $T=V|T|$ be the polar decomposition of $T$ and suppose that $\mathcal{M}(T)$ is a projection. In particular is a self-adjoint. By the preceding lemma, $V$ is a self adjoint partial isometry. Hence  $V^3=V$ and
$$\overline{\mathcal{R}(T)}=\overline{\mathcal{R}(V)}=\overline{\mathcal{R}(V^*)}=\overline{\mathcal{R}(T^*)}=\overline{\mathcal{R}(|T|)}.$$ 
Hence $V^2$ is the projection on $\overline{\mathcal{R}(V)}=\overline{\mathcal{R}(V^2)}$. 

On the other hand, we have also  $\mathcal{N}(\mathcal{M}(T))=\mathcal{N}(V)=\mathcal{N}(V^2)$. Since $\mathcal{M}(T)$ and $V^2$ are both projections with the same kernel, then 
$V^2=\mathcal{M}(T)$ and thus 
 $$V=V^3=V\mathcal{M}(T)=V(\dfrac{V|T|+|T|V}{2})=\dfrac{|T|+V|T|V}{2}.$$
From this equality,  we get  $V=\dfrac{|T|+V|T|V}{2}\geq 0$ as sum of two  positive operators.  It follows that $V$ is positive partial isometry.  In particular $V$ is also projection  and then $V=V^2$. Now from  the polar decomposition of $T$, we have   $T=V|T|=V^2|T|=|T|$ is positive. This implies that $\mathcal{M}(T)=T$ is also a projection. 
\end{proof}


\begin{lem}\label{L3} Let $T\in\B(H)$. Then
$$\mathcal{M}(T\circ P)=P \; \;  \text{for all rank one projection} \; P,\quad\mbox{if and only if}, \quad  T = I.$$

\end{lem}
\begin{proof}
Clearly, if $T=I$ then for every rank one projection $P$, we have $\mathcal{M}(T\circ P)=P$. So we need to show the direct implication. Indeed, 
let $x\in H$ be a unit vector, and $P=x\otimes x$. Let $S=T\circ P$, it is clear that $S$ is an operator of at most of rank two. So $\mathcal{M}(S)$ is at most of rank four. Using the functional trace $Tr$ well defined on the set of finite rank operator, then we get  $Tr(S)=Tr(\mathcal{M}(S))=Tr(P)=1$, thus $\langle Tx,x\rangle=Tr(T\circ P)=Tr(S)=1$. Hence we conclude that the numeral range  $\langle Tx,x\rangle =\|x\|^2$ for all $x\in H$, and thus $T=I$. 
\end{proof}

\section{ The proof of main theorem  } 
 
 The proof of  Theorem \ref{th1} will be divided in many intermediary lemmas and it  will be given at the end of this section.
  Along of this section, we assume that 
   $\Phi :\B(H)\to\B(K)$ be a bijective map, which satisfies the  condition : 
 \begin{equation}\label{c}
 \mathcal {M}(\Phi(A)\circ\Phi(B))=\Phi(\mathcal{M}(A\circ B)) \;\; \text{for all}\;\; A, B \in \B(H).
 \end{equation}

 As an immediate consequence of (\ref{c}) and Lemma \ref{L1} and Theorem \ref{L3}, we derive the following result :
 
 \begin{lem}\label{L4} The following two statements are hold :
 \begin{enumerate}[(i)]
 \item $\Phi(0)=0$,
 \item $\Phi (I)= I$.
 \end{enumerate}
 \end{lem}
 \begin{proof} 
(i).  Since $\Phi$ is onto, there is  $A\in \B(H)$ such that $\Phi(A)=0$. By  (\ref{c}) we have  
 $$0=\mathcal{M}((\Phi(0)\circ \Phi(A))=\Phi(0).$$
 Therefore, $\Phi(0)=0$.\\
(ii)  For  simplicity, let us denote $T = \Phi(I)$. 
We  take $A=B=I$ in (\ref{c}), we get that 
\begin{equation}\label{e9}
\mathcal{M}(T^2)=\Phi(I)=T.
\end{equation}
Using Theorem \ref{thm1},  we get that $T^2=T=T^*$. 

To complete the proof, it must to show $T$ is injective. Pick a   $y\in K$ such that $Ty=0$, we have also $T^*y=0$. Since $\Phi$ is onto,  there exists $B\in\B(H)$ such that
 $\Phi(B)=y\otimes y$. By  (\ref{c}), we get 
  $$0=\mathcal{M}(\frac{1}{2}(Ty\otimes y+y\otimes T^*y )= \mathcal{M} (\Phi(I) \Phi(B)) =\Phi(\mathcal{M}(B)).$$
Since $\Phi$ is bijective and $\Phi(0) = 0$,  we  have  $\mathcal{M}(B)=0$. By Proposition \ref{pp2},  we get that  $B=0$. Therefore $y\otimes y= \Phi(B)=0$ and  $y=0$. Hence $T$ is an injective projection, then $\Phi(I)=T = I$.  
 \end{proof}
 \bigskip
   As a consequence of the previous lemma, we get the following result.
    \begin{lem}\label{L5} Let $\Phi:\B(H)\to\B(K)$ be a bijective map satisfying  (\ref{c}). Then
\begin{enumerate}[(i)]
\item$\mathcal{M}(\Phi(B))=\Phi(\mathcal{M}(B))$, for all $B\in\B(H)$.\\
 In particular, $\Phi$ preserves the set of quasi-normal operators in both directions. 
\item $\Phi (A^2)=(\Phi (A))^2$ for all  $A$ quasi-normal.
\item $\Phi $ preserves the set of  orthogonal projections in both directions.
\item $\Phi $ preserves the orthogonality between the projections:
$$P\perp Q \Leftrightarrow \Phi (P)\perp\Phi (Q).$$
\item $\Phi $ preserves the order relation between  projections in both directions:
 $$Q\leq P \Leftrightarrow \Phi (Q) \leq \Phi  (P).$$
\item  $\Phi (P+Q)=\Phi (P)+\Phi (Q)$ for all orthogonal projections $P,Q$ such that $P\perp Q$. 
\item  $\Phi $ preserves the set of  rank one projections in  both directions.
\end{enumerate}
 \end{lem} 
 \begin{proof}
 (i). Taking $B=I$ in (\ref{c}), then we get $\mathcal{M}(\Phi(B))=\Phi(\mathcal{M}(B))$. Since the quasi-normal operators are exactly the fixed point of the mean transform, then  (i) holds.
 
 (ii). Let $A$ be a quasi-normal operator. Since $\Phi$ preserves the set of quasi-normal operators, then 
  $\Phi(A), \Phi(A^2)$ and $ (\Phi(A))^2$ are also quasi-normal. By (\ref{c}) with $B=A$,  we get 
  $\mathcal{M}((\Phi(A))^2)=\Phi(\mathcal{M}(A^2))$. 
  Hence  $(\Phi(A))^2=\Phi(A^2)$, it is also the fact that  the quasi-normal operators are the  fixed point of $\mathcal{M}$.
  
 (iii). It is an immediate consequence of (ii) and the fact that a   idempotent  quasi-normal  is  orthogonal projection.
 
 Throughout the remaining of the proof $P$ and $Q$ are orthogonal projections.
  
 (iv). Assume that $P,Q$ are orthogonal (i.e. $P Q = 0$ this equivalent also  to $P\circ Q=0$). Since $\Phi$ preserves the set of  orthogonal projections, then $\Phi(P), \Phi(Q)$ are orthogonal projection. By condition (\ref{c}), we get  
 $$\mathcal{M}(\Phi(P)\circ\Phi(Q))=\Phi(\mathcal{M}(P\circ Q))=\Phi(0)=0.$$
 Thus $\Phi(Q)\circ \Phi(P)=0$. The converse holds since the inverse $\Phi^{-1}$  satisfies the same condition as $\Phi$.
 
(v). Suppose that $Q \leq P$ or equivalently $Q\circ P=P \circ Q=Q$. Then by (\ref{c}), we get 
\begin{equation}\label{eq9}
\Phi(Q)\circ \Phi(P)=\mathcal{M}(\Phi(P)\circ \Phi(Q))=\Phi(Q).
\end{equation}
Since $\Phi^{-1}$ has  the same assumption as $\Phi$, then we deduce that  $\Phi$ preserves the order relation between the orthogonal projections in  both directions.

  (vi). We have $P,Q\leq P+Q$. Then $\Phi(P),\Phi(Q) \leq \Phi(P+Q)$. So $\Phi(P)+\Phi(Q)\leq \Phi(P+Q)$. Since $\Phi$ and $\Phi^{-1}$ both satisfy the same conditions, it follows that $\Phi(P)+\Phi(Q)=\Phi(P+Q)$.
 
 (vii). Let $P=x\otimes x$ be a rank one projection. Then $\Phi(P)$ is a non zero projection. Let $y\in K$ be a unit vector such that $y\otimes y \leq \Phi(P)$.  
 Thus $\Phi^{-1}(y\otimes y) \leq P$.  Since $P$ is a minimal projection and $\Phi^{-1}(y\otimes y)$ is a non zero projection, then $\Phi^{-1}(y\otimes y)= P$. Therefore  $\Phi(P)=y\otimes y$ is a rank one projection. This complete the proof.
 \end{proof}

 \begin{lem} There exists a bijective multiplicative function $h:\C\to\C$ such that :
 for every quasi-normal $A\in \B(H)$ and $\alpha \in \C$, we have 
 $\Phi(\alpha A)=h(\alpha)\Phi(A)$. 
 
 In particular, $h(0)=0$, $h(1)=1$ and $h(-1)=-1$.
 \end{lem}
 
 \begin{proof}Let $x\in H$ be a unit vector and $\alpha \in \C$. Let $T=\Phi(\alpha x\otimes x)$. By the assumption and condition (\ref{c}), we have 
  $$\mathcal{M}(\Phi(\alpha x\otimes x)\circ x\otimes x))=\Phi(\alpha x\otimes x).$$
   Therefore   
 \begin{equation}\label{ec}
 \mathcal{M}(Tx\otimes x+ x\otimes T^*x)=2T.
 \end{equation}
 Taking the norm and  using the triangular inequality, we get 
 $$2\|T\|=\|\mathcal{M}(Tx\otimes x+ x\otimes T^*x)\| \leq \|Tx\otimes x+ x\otimes T^*x\|\leq \|Tx\|+\|T^*x\|\leq 2\|T\|.$$
 Which implies the following, $$ 2\|T\|=\|Tx\otimes x+ x\otimes T^*x\| , \;\text{ and }\; \; \|T\|=\|Tx\|=\|T^*x\|.$$
Let denote $S=Tx\otimes x+ x\otimes T^*x$, form the preceding equality we get 
$$4\|T\|^2 =\|S\|^2 = \|S^*S\|=r(S^*S)\leq Tr(S^*S).$$
On the other hand, we calculate $S^*S$ then  
$$S^*S=\big{(}\|Tx\|^2x\otimes x+<x,Tx>x\otimes T^*x+<Tx,x>T^*x\otimes x + T^* x\otimes T^*x\big{)}.$$
It follows that
\begin{eqnarray*}
4\|T\|^2 &\leq & Tr(S^*S)\\
&=& Tr\big{(}\|Tx\|^2x\otimes x+<x,Tx>x\otimes T^*x+<Tx,x>T^*x\otimes x + T^* x\otimes T^*x\big{)} \\
 & = & \|Tx\|^2+ 2|<Tx,x>|^2+\|T^*x\|^2 \\*
 & = & 2\|T\|^2+2|<Tx,x>|^2\leq 4\|T\|^2.
 \end{eqnarray*}
Therefore  $$\|T\|^2\leq |<Tx,x>|^2,\;\; \text{ and } \;\; \|T\|=|<Tx,x>|=\|Tx\|=\|T^*x\|.$$
By consequence, $$\|Tx-<Tx,x>x\|^2=\|Tx\|^2-|<Tx,x>|^2=0,$$ and thus $Tx=<Tx,x>x$. 
With the same way, $T^*x=<T^*x,x>x$. By equation (\ref{ec}), it follows that 
$$\Phi(\alpha x \otimes x )=T=<Tx,x>x\otimes x=h_x(\alpha)x\otimes x, $$
where $h_x(\alpha)=<Tx,x>$ for all $\alpha \in \R$. In particular $h_x(0)=0$ and $h_x(1)=1$.

Now, by condition (\ref{c}), we have for all unit vector $x\in H$, 
$$\mathcal{M}(\Phi(\alpha I)\circ x\otimes x)=\Phi(\alpha x\otimes x)=h_x(\alpha) x\otimes x.$$ 
By Lemma \ref{L3}, $\Phi(\alpha I)x\circ x=h_x(\alpha)x \circ x$, this implies that $\Phi(\alpha I)x$ and $x$ are colinear for all unit vector $x$ of $H$. By  classical arguments, $\Phi(\alpha I)$ is a scalar multiple of the identity, so  there exists a function $h:\C\to \C$ such that  $\Phi(\alpha I)=h(\alpha)I$ for all $\alpha \in \C$. By the assumption on $\Phi$, the function $h$ is bijective and multiplicative. Moreover $\Phi(\alpha A)=h(\alpha)\Phi(A)$ for every quasi-normal operator $A$. 
 \end{proof}

\begin{lem}\label{Lem21} Let $P, Q$ two orthogonal projections and $\alpha ,\beta \in \C$, then we have 
$$\Phi(\alpha P+ \beta Q)=h(\alpha)\Phi(P)+h(\beta)\Phi(Q).$$
\end{lem} 
 \begin{proof}
 If $\alpha =0 $ or $\beta =0$, then the result follows from preceding lemma. So suppose that 
 $\alpha \ne 0$ and $\beta \ne 0$. By Lemma \ref{L5} $\Phi(P)$ and $\Phi(Q)$ are also two orthogonal projections, using condition (\ref{c}), we get the following: 
 \begin{eqnarray*}
 \mathcal{M}(\Phi(\alpha P+\beta Q)\circ \Phi(P))&=& \Phi(\mathcal{M}({\alpha}P+\beta Q)\circ P))\\
 &=& \Phi(\alpha P)=h(\alpha)\Phi(P).
 \end{eqnarray*}
 Therefore $$ \mathcal{M}(\Phi(\alpha P+\beta Q)\circ \Phi(P))=h(\alpha)\Phi(P).$$
 By Lemma \ref{Lem1}, 
 \begin{equation}
 \Phi(\alpha P+\beta Q)\circ \Phi(P)=\Phi(\alpha P+\beta Q) \Phi(P)=\Phi(P)\Phi(\alpha P+\beta Q)= h(\alpha)\Phi(P). 
 \end{equation}
 Similarly, 
  \begin{equation}
\Phi(\alpha P+\beta Q)\circ \Phi(Q)=\Phi(\alpha P+\beta Q) \Phi(Q)=\Phi(Q)\Phi(\alpha P+\beta Q)= h(\beta)\Phi(Q). 
 \end{equation}
 Again the condition (\ref{c}) implies that 
 \begin{eqnarray*}
  \Phi(\alpha P+\beta Q)&=& \Phi(\mathcal{M}((\alpha P+\beta Q)\circ ( P+ Q)))\\
   &= &\mathcal{M}(\Phi(\alpha P+\beta Q)\circ \Phi( P+ Q))\\
   & =& \mathcal{M}(\Phi(\alpha P+\beta Q)\circ (\Phi( P) + \Phi(Q))\\
   &=&  \mathcal{M}(\Phi(\alpha P+\beta Q)\circ \Phi( P) +\Phi(\alpha P+\beta Q)\circ \Phi(Q))\\
   &=& \mathcal{M}(h(\alpha)\Phi(P)+ h(\beta)\Phi(Q))\\
   &=& h(\alpha)\Phi(P)+ h(\beta)\Phi(Q). 
 \end{eqnarray*}
 \end{proof}

 \begin{lem}The function $h:\C\to\C$ is a automorphism of complex fields $\C$.
 \end{lem}
\begin{proof}
Let $x, y$ be two unit and orthogonal vectors, and let $A=x\otimes y +y\otimes x$. First, note that  $A$ is self adjoint operator  of rank  two, and we have
$$A^2=x\otimes x+y\otimes y,$$
which is a non-trivial orthogonal projection. Since the dimensional of $H$ is greater than 3, therefore  the spectrum of $A$ is  $\sigma(A)=\{-1,0,1\}$. Hence we can find $f_1, f_2$ two unit and orthogonal vectors from 
$H$ ~~($\|f_1\|=\|f_2\|=1$ and $<f_1,f_2>=0$) \: such that 
\begin{equation*}
A= f_1 \otimes f_1-f_2 \otimes f_2.
\end{equation*}
In particular $$\Phi(A)=\Phi(f_1 \otimes f_1-f_2 \otimes f_2)=\Phi(f_1 \otimes f_1)-\Phi(f_2 \otimes f_2), $$
this implies that $\Phi(A)$ is self-adjoint as sum of two rank one projection.

Now, let us consider the rank one projections $P=x\otimes x, \; Q=y\otimes y$. For  $\alpha , \beta \in \C$, we denote by  $B=\alpha P + \beta Q $.
Clearly 
$$A\circ P=\frac{1}{2} A\; \; \text{ and }\;\;A\circ Q=\frac{1}{2} A,$$   by condition (\ref{c}) we get the following equalities : 
$$h(\frac{1}{2})\Phi(A)=\Phi(\frac{1}{2}A)=\mathcal{M}(\Phi(A)\circ \Phi(P))=\Phi(A)\circ \Phi(P).$$
And 

$$h(\frac{1}{2})\Phi(A)=\Phi(\frac{1}{2}A)=\mathcal{M}(\Phi(A)\circ \Phi(Q))=\Phi(A)\circ \Phi(Q).$$

Since  $A\circ B=\dfrac{\alpha+\beta}{2} A$ then, 
\begin{eqnarray*}
h(\dfrac{\alpha+\beta}{2})\Phi(A) &=&\Phi(\mathcal{M}(A\circ B))\\
&=&\mathcal{M}(\Phi(A)\circ \Phi(B)))\\
&=&\mathcal{M}(\Phi(A)\circ \Phi(\alpha P+ \beta Q)))\\
&=&\mathcal{M}(\Phi(A)\circ (h(\alpha)\Phi(P)+ h(\beta) \Phi(Q)))\\
&=&\mathcal{M}(h(\alpha)\Phi(A)\circ \Phi(P)+h(\beta)\Phi(A)\circ \Phi(Q))\\
&=&\mathcal{M}(h(\alpha) h(\frac{1}{2}) \Phi(A)+h(\beta)  h(\frac{1}{2}) \Phi(A))\\
&=& \mathcal{M}(\big({h(\frac{\alpha}{2})+h(\frac{\beta}{2})} \big)\Phi(A))\\
&=& (h(\frac{\alpha}{2})+h(\frac{\beta}{2})) \Phi(A).
\end{eqnarray*}
Consequently $$h(\frac{\alpha}{2})+h(\frac{\beta}{2})=h(\frac{\alpha+\beta}{2}),$$ and thus
 $$h(\alpha'+\beta')=h(\alpha')+h(\beta') ,\;\;\text{ for all}\; \alpha', \beta' \in \C.$$ 

\end{proof} 
The following result has been chowed by Uhlhorn, it gives a nice characterization of a bijective map $\Psi :P_1(H)\to P_1(H)$ which preserves the orthogonality. This result can be reformulate as :
\begin{thm}[Uhlhorn's Theorem]
 Let $ \Phi: P_1(H) \to P_1(K)$
be a bijective map, with $dim ~H \geq 3$. Assume that $\Phi$ satisfies the following property 
$$PQ=0\iff \Phi(P)\Phi(Q)=0\;\; (P, Q\in P_1(H)).$$
Then there exists a unitary or anti-unitary operator $U:H\to K$, such that $\Phi$ is of the form : 
\begin{equation}\label{eq121}
\Phi(P)=UPU^*\;\;\; \text{ for all }\;\; P\in P_1(H).
\end{equation}
\end{thm}

\begin{lem}The function $h$ is the identity or the complex conjugate. Moreover for every unit vector $y\in K$ and $x\in H$ such that $\Phi(x\otimes x)=y\otimes y$, and for every self-adjoint operator $A\in \B(H)$,  we have 
\begin{equation}\label{eq120}
 \langle \Phi(A)y, y\rangle =\langle A x, x\rangle.
\end{equation}
 
 \end{lem}

 \begin{proof}
 
 Let $A\in \B(H)$ be a self-adjoint operator, for an arbitrary unit vectors $x\in H$ and $y\in K$, such that $\Phi(x \circ x)=y\otimes y$, put $P=x\otimes x$, then  $A\circ P$ is also self-adjoint of rank less than 2. And So there exists $\alpha , \beta \in \C$ and two orthogonal projections $P_1$ and $P_2$ such that $A\circ P=\alpha P_1+\beta P_2$. In particular we have $$Tr(A\circ P)=\alpha + \beta=\langle Ax,x\rangle.$$
By the Lemma \ref{Lem21}, we have 
$$\Phi(A \circ P)=\Phi(\alpha P_1+\beta P_2)=h(\alpha)\Phi(P_1)+h(\beta)\Phi(P_2),$$
 so 
$$\Phi(A \circ P)=h(\alpha)+h(\beta) = h(\alpha +\beta)=h(\langle Ax,x\rangle).$$
  On the other hand,  
 $$Tr(\Phi(A)\circ \Phi(P)=\langle \Phi(A)y, y\rangle $$
and  by the assumption we have 
 $$\langle \Phi(A)y, y\rangle =Tr(\mathcal{M}(\Phi(A)\circ \Phi(P)))=Tr(\mathcal{M}(\Phi(A\circ P)))=Tr(\Phi(A\circ P))=h(\langle Ax,x\rangle).$$
 From the preceding argument, it follows that for any self-adjoint operator $A\in \B(H)$ we have 
\begin{equation}\label{equ_image}
 h(W(A))=W(\Phi(A)).
\end{equation}
Now, let $a,b\in \R$ such that $a<b$, and let $A\in\mathcal{B}(H)$ be a self-adjoint operator such that the numerical range  $W(A)=\{\langle Au,u\rangle \;\; : \; \|u\|=1\}=[a,b]$, for example we take $A=aP+bQ$ where $P, Q$ are two orthogonal projections. 
By the preceding (\ref{equ_image})  $$ h([a,b])=h(W(A))=W(\Phi(A)),$$
 is bounded, according to  the numerical range of an operator $T\in \B(K)$ is always  bounded. Which implies that the automorphism  $h : \C\to \C$ is bounded on all segment $[a,b]\subset \R$ and also bounded in all rectangle  set $\mathcal{R}=[a,b]+i[c,d]$. By Proposition 1.1 in \cite{kallman} $h$ is either the identity $h(z)=Id_{\C}(z)=z$, for all $z\in \C$ or the conjugate complex $h(z)=\bar{Id_{\C}}(z)=\bar{z}$ for all  $z\in \C$.  
 
 In particular, we conclude that for every self-adjoint operator $A\in B(H)$ and for every unit vector $y\in K$ and $x\in H$ such that $\Phi(x\otimes x)=y \otimes y$ we have 
   $$ \langle \Phi(A)y, y\rangle =\langle A x, x\rangle   \in \R.$$
 \end{proof}
 \begin{cor}  
 The maps $\Phi : \B(H)\to \B(K)$ preserves the set self adjoint operator. Moreover, the exist an unitary or anti-unitary operator $U : H\to K$ such that     
 $$\Phi(A)=UAU^*\;\; \; \text{for all} \;\; A\in \B_s (H).$$ 
 \end{cor}
 \begin{proof}
  The map $\Phi$ satisfies the conditions of Uhlhorn's, hence 
  $\Phi$ takes the form (\ref{eq121}) on the set $P_1(H)$ of rank one operators. 
  
 Let $A\in\B(H)$ be a self-adjoint operator, and let $x \in H$ be an arbitrary  unit vectors from $H$, we have such that $\Phi(x \otimes x)=Ux\otimes U x$, and thus by equation (\ref{eq120}) 
 \begin{eqnarray}
 \langle U^*\Phi(A)Ux, x\rangle &=&  \langle \Phi(A)Ux, Ux\rangle\\
  &=&  \langle Ax, x\rangle
 \end{eqnarray}
 Therefore $U^*\Phi(A)U =A$ and this completes the proof.
 \end{proof}
In the rest of the manuscript, we replace the map  $\Phi$ by the map $T\to U^*\Phi(T)U$ which satisfies the same conditions as $\Phi$ and noted it too by $\Phi$. Then we can suppose that $\Phi(A)=A$ for every self-adjoint operator $A\in \B(H)$. To finish the proof, it must to show that $\Phi$ is the identity on $\B(H)$.

 Let $A\in\B(H)$ and let $x\in H$ be an arbitrary unit vector from $H$. 
 Put $$T=2A\circ x\otimes x=Ax\otimes x+x\otimes A^*x,$$
 thus $T^*=x\otimes Ax+A^*x\otimes x$, in particular $T$ and $T^*$ are of rank two, moreover there image 
 $$\mathcal{R}(T)\subseteq span\{x, Ax\},\;\;\text{and }\;\; \mathcal{R}(T^*)\subseteq span\{x, A^*x\}.$$
 First we have $$M(\Phi(T))=2M(\Phi(A \circ x\otimes x))=2M(\Phi(A)\circ \Phi(x\otimes x))=2M(\Phi(A)\circ x\otimes x).$$
Hence 
\begin{equation}\label{ef1}
Tr(\Phi(T))=2\langle \Phi(A)x,x\rangle
\end{equation} 
 On the other hand, we have $$T\circ (x\otimes x-\frac{I}{2})=\langle Ax,x\rangle x\otimes x.$$
 By the assumption on $\Phi$ and $h$, we get that 
$$M(\Phi(T)\circ (x\otimes x-\frac{I}{2}))=h(\langle Ax,x \rangle) x\otimes x.$$
By Lemma \ref{Lem1}, 
\begin{equation}\label{ef2}
\Phi(T)\circ (x\otimes x-\frac{I}{2})=h(\langle Ax,x \rangle) x\otimes x.
\end{equation}
Therefore, from the equation (\ref{ef2}),
\begin{eqnarray*}
\Phi(T)& = & 2\Phi(T)\circ (x\otimes x)-2h(\langle Ax,x \rangle ) x\otimes x\\
& = & \Phi(T)x\otimes x +x\otimes\Phi(T)^*x-2h(\langle Ax,x \rangle) x\otimes x.
\end{eqnarray*}
Multiplying the last equation on the right by $x\otimes x$, then we get 
$$\Phi(T)x \otimes x = \Phi(T)x\otimes x +\langle \Phi(T)x, x\rangle x\otimes x-2h(\langle Ax,x \rangle ) x\otimes x.$$
This implies the following equation 
\begin{equation}\label{ef3}
\langle \Phi(T)x, x\rangle=2h(\langle Ax,x \rangle).
\end{equation}
Now, from equation  (\ref{ef2})  we get 
\begin{eqnarray*}
 h(<Ax,x>)  & = & Tr(\Phi(T)\circ (x\otimes x-\frac{I}{2})\\
 &=&Tr(\Phi(T)(x\otimes x-\frac{I}{2}))\\
 & = & \langle \Phi(T)x,x\rangle -\dfrac{Tr(\Phi(T))}{2}.
 \end{eqnarray*}
Then, from equations  (\ref{ef1}) we get 
\begin{equation}\label{ef4}
Tr(\Phi(T))=2\langle \Phi(T)x,x\rangle-2h(\langle Ax,x\rangle) =2h(\langle Ax,x\rangle).
\end{equation}
 From equation (\ref{ef1}) and (\ref{ef4}), we conclude that
 \begin{equation}\label{ef5}
 \langle \Phi(A)x,x\rangle=h(\langle Ax,x\rangle).
 \end{equation} 
 The last equality holds for every unit vector $x\in H$ and every $A\in \B(H)$.  Let us distinct the following two cases : 
 
\textbf{Case 1 : } If the function $h=Id_{\C}$, then from the equation (\ref{ef5}) we have, 
 \begin{equation}\label{ef6}
 \langle \Phi(A)x,x\rangle=\langle Ax,x\rangle,\text{ for all unit vector $x$ and $A\in \B(H)$}.
 \end{equation} 
 This  show that the map $\Phi$ is the identity. 
 
 \textbf{Case 2 : } Suppose that the function $h=\overline{Id}_{\C}$ is the complex conjugate. Hence from equation (\ref{ef5}), we get 
  \begin{equation}\label{ef6}
 \langle \Phi(A)x,x\rangle=\langle A^*x,x\rangle,\text{ for all unit vector $x$ and $A\in \B(H)$}.
 \end{equation} 
 Which implies that $\Phi(A)=A^*$ for all $A\in \B(H)$. 
 
 Now, let us consider  $A=x\otimes x'$ with $x, x'$  are  unit,  independent and non-orthogonal vectors  in $H$. Then $A^*=x'\otimes x$. By Proposition \ref{p1}, we have 
 $$ \mathcal{M}(\Phi(A))=\mathcal{M}(A^*)=\frac{1}{2}(x'+<x', x>x)\otimes x),$$
 and 
 $$\Phi(\mathcal{M}(A))=(\mathcal{M}(A))^*=\frac{1}{2}(x'\otimes (x+<x', x>x')),$$
which contradicts with the fact that $\Phi$ commute with the mean transform. 
 Hence, we conclude that the function $h$ must be the identity, and $\Phi$ must be linear and of the form 
 $$\Phi(T)=UTU^*, \;\;\text{ for all } T\in \B(H),$$ where $U:H\to K$ is a unitary or anti-unitary operator.
  \vspace*{0.8cm}
  
  {\bf  Acknowledgments.}

 I wish to thank Professor Mostafa Mbekhta  for the  interesting discussions as well as his useful suggestions for the improvement of this paper.


\begin {thebibliography}{99}

\bibitem {alu}    {\sc A. Aluthge},
\emph{} { \textit {On p-hyponormal operators for $0 < p <1$}}, Integral Equations
Operator Theory 13 (1990), 307-315. 

\bibitem{bmg} {\sc F. Botelho ; L. Moln\'ar ; G. Nagy},
\emph{}{ \textit { Linear bijections on von Neumann factors commuting with $\lambda$-Aluthge transform}}, 
Bull. Lond. Math. Soc. 48 (2016), 74-84. 

\bibitem{cha} {\sc F. Chabbabi},
\emph{}{\textit {Product  commuting maps  with the  $\lambda$-Aluthge transform}}, J. Math. Anal. Appl. 449 (2017), 589-600.


\bibitem{fm} {\sc F. Chabbabi; M. Mbekhta},
\emph{}{\textit { Jordan product  commuting nonlinear maps  with the  $\lambda$-Aluthge transform}},  J. Math. Anal. Appl. 450 (2017),  293-313. 

\bibitem{ccm} {\sc F.Chabbabi, R. E. Curto,  M. Mbekhta}
\emph{}{\textit {The mean transform and the mean limit of an operator}}, Proc. AMS, 2019, 147(3), pp. 1119-1133.  

\bibitem{kp2} {\sc F. Chabbabi and M. Mbekhta},
\emph{}{ \textit { Polar decomposition, Aluthge and mean transforms, Linear and Multilinear Algebra and Function Spaces}}, Contemp. Math., Proc., Amer. Math. Soc., 2020, 750, 89-107. 

\bibitem{fur} {\sc T. Furuta},
\emph{}{\textit {Invitation to linear operators}}, Taylor  Francis,  London 2001.

\bibitem{LLL} {\sc S.H. Lee, W.Y. Lee and J. Yoon},
\emph{}{ \textit { The mean transform of bounded linear operators}}, J. Math. Anal. Appl. 410 (2014), 70-81.

\bibitem{kp3} {\sc  I. Jung, E. Ko, and C. Pearcy },
\emph{}{ \textit { Aluthge transform of operators}}, Integral Equations Operator Theory 37 (2000), 437-448.





\bibitem{jkp}  {\sc  I.B. Jung, E. Ko and S. Park}, \emph{}{\textit {Subscalarity of operator transforms}},  Math. Nachr. 288 
 (2015), 2042--2056.


\bibitem{kallman} {\sc R. Kallman, R. Simmons},
\emph{}{\textit { A theorem on planar continua and an application to automorphisms of the
field of complex numbers}}, Topology and its Applications 20 (1985), 251-255
\bibitem{kat} {\sc T. Kato},
\emph{}{\textit {Perturbation Theory for Linear Operators }}, Springer, Berlin 1980.


\bibitem{hyy}{\sc S. Lee, W. Lee, and J. Yoon}, 
\emph{}{ \textit {The mean transform of bounded linear operators}}, J. Math. Anal. Appl. 410 (2014), 70-81.



\bibitem{oku} {\sc K. Okubo}, 
\emph{}{ \textit {On weakly unitarily invariant norm and the  Aluthge  transformation}}, Linear Algebra Appl. 371 (2003),  369-375.


\bibitem{sem} {\sc P. Semrl}, 
\emph{}{ \textit {Linear mapping preserving square-zero matrices}}, Bul. Austrl. Math. Soc.  48 (1993),  365-370.


\bibitem{uu} {\sc U. Uhlhorn}, 
\emph{}{ \textit {Representation of symmetry transformations in quantum mechanics}}, Ark. Fysik 23 (1962), 307-340.


\bibitem{yam} {\sc T. Yamazaki},
\emph{}{\textit {An expression of the spectral radius via Aluthge tranformation}}, Proc. Amer. Math. Soc. 130 (2002), 1131-1137.

\end{thebibliography}
\end{document}